\documentclass{article}

\usepackage{amsmath,amsthm,amssymb,graphicx}

\newtheorem{Theorem}{Theorem}
\newtheorem{Lemma}{Lemma}
\newtheorem{Definition}{Definition}
\newtheorem{Conjecture}{Conjecture}
\newtheorem{Corollary}{Corollary}

\newcommand{\Dq}{\mathrm{maxdeg}_q}
\newcommand{\Da}{\mathrm{maxdeg}_a}

\newcommand{\la}{\lambda}
\newcommand{\ga}{\gamma}
\newcommand{\Ga}{\Gamma}
\newcommand{\rot}{\mathrm{rot}}
\newcommand{\lag}{\langle}
\newcommand{\rag}{\rangle}

\newcommand{\e}{\epsilon}

\title{The degree of the colored HOMFLY polynomial}
\author{Roland van der Veen}

\begin{document}

\maketitle

\abstract{The colored HOMLFY polynomial is an important knot invariant depending on two variables $a$ and $q$. 
We give bounds on the degree in both $a$ and $q$ generalizing Morton's bounds \cite{Mo86} for the ordinary HOMFLY polynomial.
Our bounds suggest that the degree detects certain incompressible surfaces in the knot complement and perhaps more generally
features of the $SL(N)$ character varieties of the knot group. We formulate a precise conjecture along these lines generalizing the slope 
conjecture of Garoufalidis \cite{Ga11}. We prove our conjecture for all positive knots.

Our technique is a reformulation of the MOY state sum \cite{MOY98} using $q$-analogues of Ehrhart polynomials.
As a direct application we explicitly compute the $r$ coefficients of $r$-colored HOMFLY polynomial of any positive braid. 
}

\section{Introduction}

The colored HOMFLY polynomial is an important invariant in knot theory. It has many intriguing connections to string theory contact homology
integrable systems Gromov-Witten invariants. Also much effort was directed to understanding its categorification. 
Given all this interest it is rather surprising that not much seems to be known about some of its simplest properties such as its \emph{degree}. 
Even simple bounds on the maximum degree seem to be unavailable in the literature. The purpose of this paper
is to provide such bounds and to show that the degree is actually an interesting invariant already.

The version of the colored HOMFLY polynomial we are discussing unifies all the Reshetikhin-Turaev quantum $\mathfrak{sl}_N$ invariants of knots colored with
symmetric powers of the standard representation. Our notation is $P_r(L;a,q)$ for the unreduced $r$-colored HOMFLY polynomial of 
$L$ where each component of $L$ is colored by the $r$-th symmetric power.
Setting $a = q^N$ yields the corresponding $\mathfrak{sl}_N$ invariant. In particular we obtain the $r$
colored Jones polynomial for $N=2$.

Despite its name $P_r(L;a,q)$ is not exactly a two variable polynomial. It is rather a rational function in $q$ and $a$. As such it still
makes sense to talk about its maximal degree in $q$. Indeed we can always expand in a Laurent series in $q^{-1}$. By definition this has a highest degree
term in $q$ and this then is the maximal degree in $q$, notation $\Dq$. The same goes for the variable $a$. 

\begin{Theorem}
\label{Thm.Bounds}
Let $D$ be any oriented link diagram $D$. The number of positive and negative crossings are denoted $c_\pm$ and the number 
of positively/negatively oriented Seifert circles is $s_\pm$. 
\begin{enumerate}
\item[a.] $\Da P_r(D;a,q)\leq \frac{r}{2}(-c_+ + c_- + s_+ + s_-)$
\item[q.] $\Dq P_r(D;a,q)\leq \frac{r}{2}(s_+ - s_- + c_+ + c_-(2r-1))$
\item[p.] For positive diagrams $\frac{r}{2}(-s_+ -s_- + c_+) \leq \Dq P_r(D;a,q)$
\end{enumerate}
\end{Theorem}

We get similar bounds for the the minimal degree by considering the mirror image $\bar{D}$ of the diagram $D$, because
$P_r(\bar{D};a,q)=P_r(D;a^{-1},q^{-1})$. These bounds also apply to the anti-symmetric version of the colored HOMFLY polynomial
because of the formula $P_{r^t}(a,q)=(-1)^rP_{r}(a,q^{-1})$ \cite{Zh13}.

These bounds should be compared to Morton's bounds \cite{Mo86} in the case $r=1$. In our notation he
proved that $\Da P_r(D;a,q)\leq \frac{1}{2}(-c_+ + c_- +s_+ + s_-)$ in perfect agreement with our bound.
For the $q$-degree Morton's upper bound is $\Dq P_r(D;a,q)\leq \frac{1}{2}(-s_+-s_- + c_+ + c_-)$. This is much sharper than our general $q$-bound and matches
our lower bound in the positive case. We conjecture there actually is equality for all positive diagrams.

Let us illustrate our bounds in the simple case of the closure of a $2$-braid $T(c,2)$ with $c$ crossings, $c\in \mathbb{Z}$ \cite{GNS15}.
$P_r(T(2,c);a,q) = $

\begin{equation}
\label{Eq.T2}
(q^{\frac{r(r-1)}{2}}a^{\frac{r}{2}})^{-c}\sum_{k=0}^rq^{\frac{c((r-k)^2-k)}{2}}\frac{(-1)^{c(r-k)}a^rq^{-k-r}(qa^{-1};q^{-1})_k(a^{-1};q^{-1})_{2r-k}}{(q^{-1};q^{-1})_k(q^{-1};q^{-1})_{2(r-k)}(q^{-2(r-k+1)};q^{-1})_k}
\end{equation}

First the maxdegree in $a$. All crossings are of the same sign so since $-c=-c_+ + c_-$. For any $k$ it is obtained we get terms of $a$-degree $\frac{r}{2}(-c+2)$ attaining our upper bound since $s=2$
Next the $q$. For $c\geq0$ we have $c=c_+$ a positive diagram with $s=2$. Looking at the $k=0$ term we find the maximal $q$ degree of $-c\frac{r(r-1)}{2}+c\frac{r}{2}-r = \frac{r}{2}(c_+ - s)$   
For $c<0$ our bound is very poor. In the formula the maxdegree is obtained by looking at the $k=r$ term. We find $-c\frac{r^2}{2}-2r+1$. This is about half of the given upper bound since $c=-c_-$.

\subsection{Topological interpretation of the $q$-degree}

As an application for our bounds we give a topological interpretation of the growth rate of the $q$-degree. For simplicity
we restrict ourselves to the knot case, see \cite{V13} for a similar conjecture in rank $1$ for links. 

Recall the knot exterior $E_K$ is defined as $\mathbb{S}^3-N(K)$, where $N(K)$ is an open tubular neighborhood of the knot $K$. $E_K$ is a
compact $3$-manifold with torus boundary and a canonical choice of a basis $\lambda,\mu$ $H_1(\partial E_K)$. A properly embedded
essential surface $(\Sigma,\partial \Sigma)\subset (E_K,\partial E_K)$ has boundary slope $p/q$ if $\partial \sigma = p\mu + q\lambda \in H_1(\partial E_K)$.
Hatcher showed every knot has only finitely many slopes \cite{Ha82}. For Montesinos knots there is an algorithm to compute these slopes.
In particular it shows that any rational number can be the slope of some knot. Culler and Shalen showed that some slopes
are the slopes of the Newton polygon of the A-polynomial of the knot. 

\begin{Definition}
For a knot $K$ define the set of HOMFLY slopes $\mathcal{S}(K)$ as follows
$\mathcal{S} = \{r^{-2}\Dq P_r(K;a,q)|r\in \mathbb{N}\}'$
where $X'$ means the set of limit points (accummulation points) of any set $X\subset \mathbb{R}$.\\
Also let $\mathcal{B}(K)$ denote the set of boundary slopes of essential surfaces as defined above.
\end{Definition}

Motivated by the AJ conjecture and the slope conjecture \cite{Ga11} for the 
colored Jones polynomial, i.e. the specialization $a=q^2$, we propose

\begin{Conjecture}(HOMFLY slope conjecture)\\
For any knot $K$ we have $4\mathcal{S}(K)\subset \mathcal{B}(K)$
\end{Conjecture}
 
Since positive knots $K$ are fibered and the fiber surface is essential \cite{St78} with slope $0$ our
degree bounds immediately imply $\mathcal{S}(K) = \{0\}$ and hence:

\begin{Corollary}
The HOMFLY slope conjecture is true for all positive knots.
\end{Corollary}

Further evidence is provided by our negative $2$-braid examples. We have seen that the $\Dq = \frac{c_-r^2}{2}+1$.
Therefore $\mathcal{S} = \{\frac{c_-}{2}\}$. For any alternating knot $2c_-\in \mathcal{B}$, see for example \cite{Ga11} so the HOMFLY slope conjecture holds.  

Comparing to the original slope conjecture it is likely that there are knots $K$ such that $\mathcal{S}(K)$ contains fractions with high denominators \cite{GV14}.
This gives an interesting constraint on any state sum for the colored HOMFLY polynomial: Either it involves a huge amount of cancellation or
it explicitly encodes such fractions.

\subsection{The head of the colored HOMFLY}

An attractive way to organize the colored HOMFLY polynomial is as a Laurent series 
\[P_r(K;a;q)=\sum_{j=d_r}^\infty c^r_j(a) q^{-j}\]
This goes one step beyond our investigation of the maximal $q$-degree.
For positive diagrams and $r= 1$ the leading term of such an expansion was called the top of the HOMFLY polynomial \cite{KM13}
(using the variable $z=q^{\frac{1}{2}}-q^{-\frac{1}{2}}$).

The above expansion is useful because the coefficients $c_j^r(a)$ seem to stabilize once scaled properly.
Similar stabilization phenomena were discussed for positive and alternating knots in the colored Jones case \cite{AD13,GL11}.
Although unproven this stabilization is likely to persist for general knots after one passes to properly chosen
arithmetic subsequences in $r$.

As a working definition let us define the head of the colored HOMFLY as a rescaled version of the top $r$ coefficients. 

\begin{Definition}
The \emph{head} of the colored HOMFLY polynomial is defined as 
\[a^{f_r}\sum_{j=0}^{r}c^r_{j+d_r}(a)q^{-j}\]
where $a^{f_r}$ is a monomial chosen such that the leading term of $a^{f_r}c_{d_r}(a)$ is $1$.
\end{Definition}

The head is expected to be independent of $r$ in the following sense. Let $H(a,q)\in \mathbb{Z}[a,a^{-1}][[q^{-1}]]$ be a power series such that
the highest $r$ coefficients in $q$ are equal to the head $a^{f_r}\sum_{j=0}^{r}c^r_{j+d_r}(a)q^{-j}$ as before. 
As mentioned this definition needs to be adjusted by passing to subsequences for general knots. For positive knots however it is fine.

\begin{Theorem}
For any positive braid closure, the head and equal to the head of the unknot, i.e. 
equal to the first $r$ coefficients of
\[\frac{(a^{-1};q)_{\infty}}{(q^{-1};q^{-1})_{\infty}}\]
\end{Theorem}

The proof of this theorem shows the power of our new state sum in thinking about rather general situations in a 
diagrammatic fashion. The idea is to reduce to the case of $2$-braids. Recall that for $2$-braids we have an explicit formula
\eqref{Eq.T2}. For $c\geq 0$ the head is provided by the term $k=0$ only. It is then clear from the formula that the head is
independent of $c$ and given by $(a^{-1};q)_r$. Note that this family includes the unknot as $c =1$.

The same formula for $c<0$ shows that the notion of head is not vacuous. In this case the head comes from the $k=r$ term only
and can be expressed by the power series
\[\frac{(a^{-1};q)_\infty^2}{(q^{-1};q^{-1})_\infty(q^{-2};q^{-1})_\infty}\]

\subsection{Plan of the paper}

In the first section we develop a brief theory of $q$-Ehrhart polynomials that is perhaps interesting in its own right. 
The $q$-Ehrhart polynomials play an essential role in our reformulation of the MOY state sum for the colored HOMFLY polynomial.
The next two sections apply the state sum to derive our degree bounds and compute the head of the colored HOMFLY polymomial.
Finally we prove our state sum by showing how it relates to the original MOY approach. We end with a discussion.

\subsubsection{Acknowledgements}

The author thanks Tudor Dimofte, Hiroyuki Fuji, Stavros Garoufalidis, Paul Gunnells and Satoshi Nawata for stimulating conversations 
and Satoshi for his permission to use his unpublished formulas for torus links.
The author wishes to thank the organisers of the Oberwolfach workshop on low-dimensional topology and number theory for providing
an opportunity to present a of this work.
The author was supported by the Netherlands organization for scientific research (NWO).

\section{q-Ehrhart polynomials}
\label{sec.Ehrhart}

We present here a brief account of a $q$-analogue of Ehrhart polynomials and prove a reciprocity theorem for them. These polynomials provide
a flexible generalization of Gaussian $q$-binomial coefficients suitable for expressing HOMFLY polynomial of complicated knots.

Our treatment relies on the standard theory of generating functions for lattice point ennumeration and Stanley reciprocity, see for example \cite{BR07}.
The present development of $q$-Ehrhart polynomials is heavily inspired by the preprint by Chapoton \cite{Ch13}. The special case where
the polytope is generic and the linear form positive so that no polymomials in $N$ appear is due to Chapoton. 

For a polytope $Q\subset\mathbb{R}^m$ and a linear form $\lambda$ we consider the weighted sum 
\[W_\lambda(Q,q)=\sum_{x  \in \mathbb{Z}^m \cap N Q}q^{\lambda(x)}\]
We would like to study how $W$ changes as $Q$ gets scaled as $N Q$. For example when $Q = [0,1]$ then
\[W_\lambda(NQ,q) = \sum_{x = 0}^N q^{\lambda x} =\begin{cases} \frac{1-q^{N\lambda}}{1-q^\lambda} &\mbox{if } \lambda\neq 0 \\ N &\mbox{if } \lambda=0 \end{cases}\] 

\begin{Theorem} {\rm($q$-Ehrhart theorem):}\\
\label{q.Ehrhart}
Let $Q$ be a lattice polytope and a linear form $\lambda$.
\begin{enumerate} 
\item[a.] There exists a two variable Laurent polynomial $E_{\lambda,Q}(a,b,q)\in \mathbb{Q}(q)[a^{\pm 1},b]$ such that for every $N\in \mathbb{N}$:
\[
E_{\lambda,Q}(q^N,N,q) = W_\lambda(NQ,q)
\]
\item[b.] Denote by $(NQ)^o$ the interior lattice points of the polytope then we have the following form of {\bf Ehrhart reciprocity}: 
\[
E_{\lambda,Q}(q^N,-N,q^{-1}) = (-1)^{\mathrm{dim}(Q)}W_\lambda((NQ)^o,q)
\]
\end{enumerate}
\end{Theorem}

\begin{proof}
By triangulating and the inclusion-exclusion principle we may reduce to the special case where $Q$ is a (lattice) simplex \cite{BR07}. For such a simplex $Q$ we consider the generating series 
$G_{\lambda,Q}(q,z) = \sum_{N=0}^\infty W_{\lambda}(Q,q)z^N$. By Theorem 3.5 on the integer point transform \cite{BR07} we know that $G(q,z)$ is a rational function in $q,z$ whose denominator can be factored as $\prod_{w\in V(Q)} (1-q^{\lambda(w)} z)$. Here the product runs over the set of vertices $V(Q)$ of the simplex $Q$.
If the denominator had only one such factor then we could expand $(1-q^{r} z)^{-s} = \sum_N (N+s)\hdots (N+1)q^Nz^N$ to obtain part 1) of the theorem. We can reduce to this case by expanding $G_{\lambda,Q}(q,z)$ into partial fractions with respect to $z$ with coefficients in $\mathbb{Q}(q)$.

For the reciprocity statement of part 2) we apply Stanley's reciprocity theorem, Theorem 4.3 in \cite{BR07}, to the generating function
\[\sum_{N=0}^\infty W_{\lambda,Q^o}(q)z^N=G_{\lambda,Q^o}(q,z)\] 
to obtain
\[(-1)^{1+\mathrm{dim}(Q)}G_{\lambda,Q^o}(q,z) = G_{\lambda,Q}(q^{-1},z^{-1})\] 
\[= \sum_{N=0}^\infty E_{\lambda,Q}(q^{-N},N,q^{-1})z^{-N}= \sum_{N=-1}^{-\infty} E_{\lambda,Q}(q^{N},-N,q^{-1})z^{N} \]
\[=\sum_{N=0}^{\infty} E_{\lambda,Q}(q^{N},-N,q^{-1})z^{N}\]
In the last equality we used the fact that the left hand side is a rational function in $(q,z)$ of the special form $\sum_{i,j}c_{ij} (1-q^{r_i} z)^{-s_j}$ with $c_{ij}\in \mathbb{Q}(q)[z]$.

\end{proof}

The relevant examples for our state sums are the order polytopes \cite{St72}. Given a partial ordered set $(X,\leq)$ we consider the \emph{order polytope} $P_X \subset \mathbb{R}^{X}$ given by 

\[
P_X = \{v\in \mathbb{R}^X | x\leq x' \Rightarrow v_x\leq v_x'\} \cap [0,1]^X
\]

The order polytope is a lattice polytope with vertices in the unit cube $[0,1]^X$.
The order polytope of a linear ordering is a simplex. The set of all linear orderings extending the partial order on the set $X$ can thus be interpreted as a triangulation of $P_X$.

For example consider the order $X=\{1,2\}$ then $P_X$ is the triangle $v_1\leq v_{2} \in [0,1]^2$. We choose the linear form $\lambda(v_1,v_2)=v_1-v_2$. The generating function for the Ehrhart polynomials is 
\[G_{\lambda,Q}(q,z) = \frac{1}{(1-z)^2(1-q^{-1} z)}=\]
\[\frac{q}{(-1 + q) (-1 + z)^2} - \frac{q}{(-1 + q)^2 (1 - z)} + 
\frac{1}{(-1 + q)^2 (1 - q^{-1} z)}\]
Hence the Ehrhart polynomial is $E_{\lambda, P_X}(a,b,q) = $

\[
\frac{q(b+1)}{(-1 + q)} - \frac{q}{(-1 + q)^2} + 
\frac{a^{-1}}{(-1 + q)^2} = \frac{a^{-1} - 2 q - b q + q^2 + b q^2}{(1-q)^2}
\]
And indeed summing the powers of $q$ in $3P_X$ gives $E_{\lambda, P_X}(q^3,3,q) = 4+3q^{-1}+2q^{-2}+q^{-3}$ as expected.
According to the reciprocity theorem the interior points are obtained by $E_{\lambda, P_X}(q^3,q,-3,q^{-1}) = q^{-1}$ again as expected. 

\begin{Lemma}
For any partial ordered set $X$ positive linear form $\lambda$ on $\mathbb{R}^X$, the $q$-Ehrhart polynomial $E_{\lambda,P_X}(a,b,q)$ does not depend on the variable $b$.  
\end{Lemma}
\begin{proof}
Using the above triangulation we may reduce to the case where $X$ is a linearly ordered set, so that $P_X$ is a simplex. The coordinates of each vertex $v$ of the simplex are in bijection with the elements of $x$ as follows. For $x\in X$ define $v^x$ as $v^x_{x'} = 0$ if $x'\leq x$ and $v^x_{x'}=1$ otherwise. It is then clear that $\lambda$ takes different values on each of the vertices. This means that the generating series from the proof of the theorem 
$G_{\lambda,P_X}(q,z) = \sum_{N=0}^\infty W_{\lambda}(P_X,q)z^N$ 
has a denominator that is a product of distinct factors $\prod_{x\in X}^p(1-q^{\lambda{v_x}} z)$
Expanding in partial fractions now shows that the only $N$ dependence of the coefficients of the series $G$ is as $a=q^N$.
\end{proof}

As a preparation for our degree bounds for the HOMFLY polyomial we derive some elementary bounds on the degree in
$a$ and $q$ for the $q$-Ehrhart polynomials.

\begin{Lemma}
\label{Lem.EhrhartBounds}
For any polytope $Q$ with no interior vertices and linear form $\la$ we have
\begin{enumerate}
\item[a.] $\Da E_{Q,\la}(aq^{-1},-1,q) \leq \max_{x\in Q} \la(x)$
\item[q.] $\Dq E_{Q,\la}(aq^{-1},-1,q) \leq -\max_{x\in Q} \la(x)$
\end{enumerate}

\end{Lemma}
\begin{proof}
As usual by inclusion-exclusion we may reduce to the case where $Q$ is a simplex. Recall that in this case
the Ehrhart series is given by 
\[
G_{\lambda,Q} = \prod_{w\in V(Q)} (1-q^{\lambda(w)} z)^{-1} 
\]
After expanding in partial fractions the highest degree term in $a$ will come from the
the maximum of $\la$ over the vertices of $Q$, proving part $a$.

The proof of part $q$ is similar but we need to work a little more since the coefficients of the partial fraction expansion are 
rational functions in $q$. Collect like terms in $G$ to write
\[G_{\la,Q} = \prod_{j=0}^{k}(1-q^{e_j}z)^{-m_j}\]
where $m_j\in \mathbb{N}$ and we may assume that $e_k>e_j$ for $k<j$. Expanding in partial fractions yields
\[G_{\la,Q} = \sum_{j=0}^k\sum_{i=1}^{m_j} C_{ij}(q)(1-q^{e_j})^{-i}\]
The reader is invited to check that we have $\Dq C_{ij}(q) \leq e_j-e_0$. Finally the term $C_{ij}(1-q^{e_j})^{-i}$ contributes the power $a^{e_j}$ to $E_{\la,Q}(q,a,b)$
so that the replacement $a\mapsto aq^{-1}$ decreases the $q$-degree by $e_j$ leaving us with $-e_0 = -\max_{x\in Q} \la(x)$.
\end{proof}

\section{The symmetric MOY state sum}
\label{Sec.Sym}

In this section we describe our new state sum for any link diagram $D$ in terms of MOY graphs.
A MOY graph is a pair $(\Ga,\ga)$, where $\Ga$ is an oriented trivalent graph embedded in the plane without sources or sinks together with a flow $\ga$.
A flow is a function $\ga:E(\Ga)\to \mathbb{Z}_{\leq 0}$ such that at every vertex the sum of the values of the incoming edges equals the sum of the outgoing edges. 
We want to define the symmetric state sum evaluation $[\Ga,\ga](a,q)$ of any MOY graph. For this we need a couple of definitions. 

\begin{Definition}
\begin{enumerate}
\item An elementary flow $\e$ is a non-zero flow that takes values in $\{0,1\}$. The set of such flows is denoted $\mathcal{E}$. 
\item The orientation of $\Ga$ induces a rotation number on each component of $\e^{-1}\{1\}$ (i.e. $+1$ for counter-clockwise $-1$ for clockwise). 
The rotation number $\rot(\e)$ is the sum of the rotation numbers of the components of $\e$
\item The intersection number of a pair of elementary flows is defined by the formula $\lag\delta,\e\rag=$
\[
\frac{1}{4}|\{v\in V(\Ga)| \delta(v_L)=1, \e(v_R)=1\}|-\frac{1}{4}|\{v\in V(\Ga)| \e(v_R)=1, \delta(v_L)=1\}|
\]
where $v_L$ is the edge at vertex $v$ that goes left respect to the orientation and $v_R$ goes right at $v$.
\end{enumerate}
\end{Definition}

We are now ready to define the symmetric state sum for a MOY graph. This may be regarded as a generalization of the HOMFLY polynomial to
trivalent planar graphs. Including crossings in the usual way would give a generalization to any trivalent graph in the three sphere.

\begin{Definition}
The symmetric evaluation $[\Ga,\ga](a,q)$ of a MOY graph $(\Ga,\ga)$ is defined as
\[
[\Ga,\ga](a,q) = \sum_{\e: \sum_j \e_j = \gamma} (-1)^{|\e|}\left(\frac{q}{a}\right)^{\frac{1}{2}\sum_{j}\rot(\e_j)}
q^{-w(\e)} E_{\rot,P_{\e}}(\frac{a}{q},-1,q)
\]
Here the sum is over sequences $\e=(\e_1,\e_2\hdots)$ of elementary flows and $P_\e$ is the order polytope
of $\e$ interpreted as a linearly ordered set $\e_1\leq\e_2\leq \hdots$. The length of a sequence $\e$ is denoted $|\e|$ 
and finally the weight is $w(\e) = \sum_{i<j}\lag \e_i,\e_j \rag$.
\end{Definition}

\begin{figure}[htp]
\label{fig.Di}
\begin{center}
\includegraphics[width=7cm]{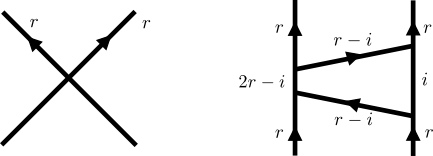}
\end{center}
\caption{Replacing a crossing (left) by a piece of a MOY graph (right).}
\end{figure}

The main result of this section is the following expression for the colored HOMFLY of any link.

\begin{Theorem}
\label{Thm.SymMOY}
Let $L$ be a link with oriented link diagram $D$ and $C$ the set of its crossings, $c_+$ positive and $c_-$ negative ones.
Expand the diagram as a linear combination of MOY graphs $D_i$ by replacing all the crossings as shown in Figure \ref{fig.Di}, then
$P_{r}(D;a,q) =$ 
\[(-a^{-\frac{r}{2}} q^{-\frac{r(r-1)}{2}})^{c_+-c_-} 
\sum_{i\in \{0,1,\hdots r\}^{C}}\left(\prod_{c\in C}(-1)^{i_c} q^{-\sigma(c)\frac{i_c}{2}}\right)[D_i](a,q)\]
where the sum runs over $i = (i_c)_{c\in C}$ and $\sigma(c)$ is the sign of the crossing $c$.
\end{Theorem}

The proof of this theorem can be found in section \ref{Sec.Reform}. The idea behind the state sum is to reformulate the original MOY state sum \cite{MOY98}

\subsection{Example: Hopf link}
As a simple illustration we evaluate our state sum on the positive Hopf link $H$ colored by $1$ so we compute $P_1(H;a,q)$.

The set of crossings is $C=\{1,2\}$ and $c_+=2$, $c_-=0$.
Expanding the two crossings as in Figure \ref{fig.Di} we get four terms. 

\[P_{1}(H;a,q) = a^{-1}\sum_{i\in \{0,1\}^2}\left(\prod_{c=1}^2(-1)^{i_c} q^{-\frac{i_c}{2}}\right)[D_i](a,q)\]

\begin{figure}[htp]
\label{fig.Hopf}
\begin{center}
\includegraphics[width=9cm]{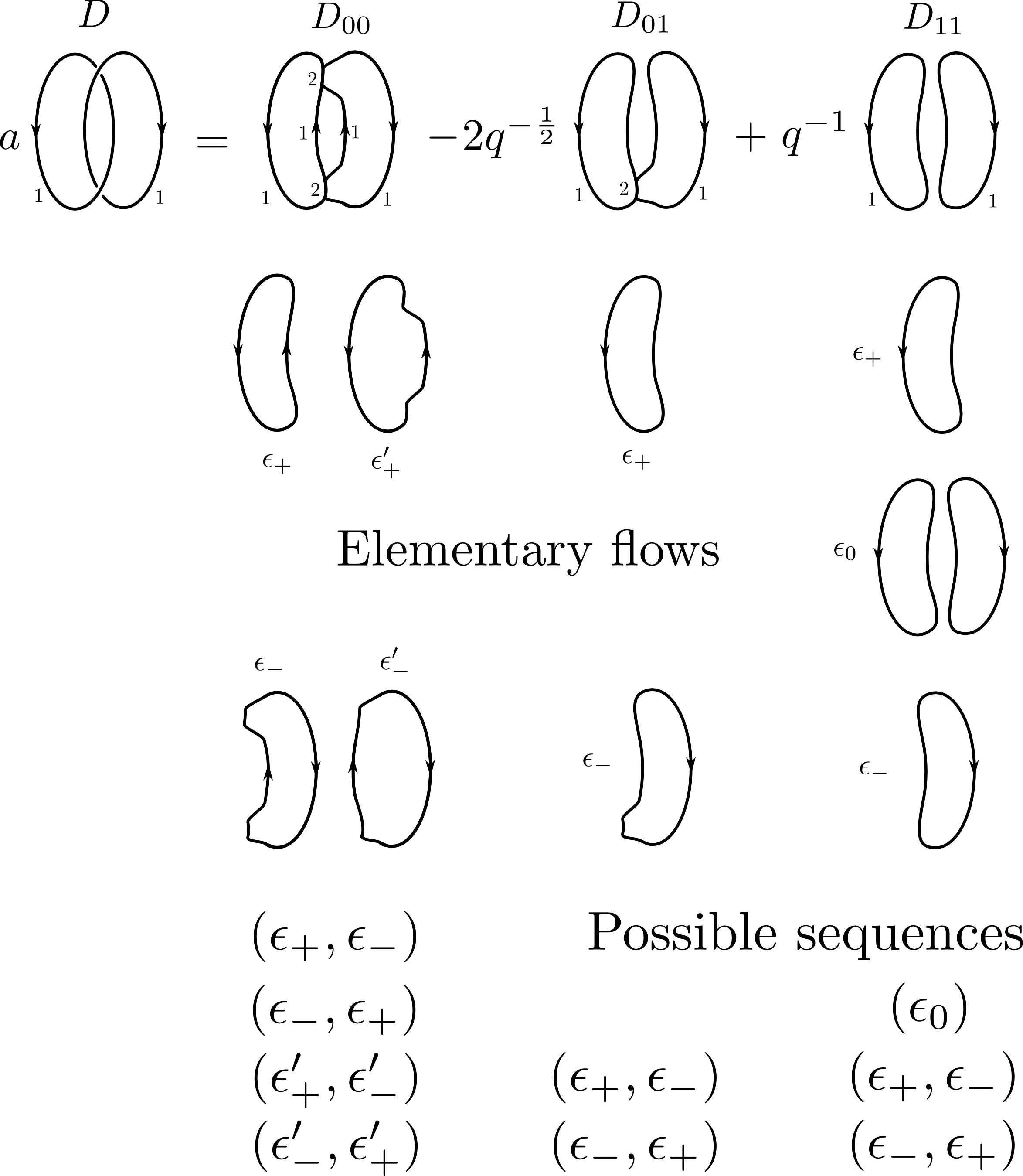}
\end{center}
\caption{First row: Expanding a diagram for the Hopf link into MOY graphs.
Under each MOY graph on the top row all the elementary flows necessary to calculate it
have been drawn. In the last row the possible combinations of elementary flows are given.}
\end{figure}

Since $D_{01}$ and $D_{10}$ we have only drawn one of them in Figure \ref{fig.Hopf}, with coefficient 2.
Below each MOY graph the same figure shows the possible elementary flows and also all possible
sequences of these flows that add up to the coloring. The subscript of each elementary flow is its
rotation number. For example $D_{11}$ allows three different elementary flows: $\e_+,\e_-$ and the two component flow $\e_0$.
The subscript indicates the rotation number of each flow. As shown there are only three possible sequences of elementary flows 
that add up to the flow of $D_11$, namely $(\e_+,\e_-), (\e_+,\e_-)$ and $(\e_0)$. These will be the three terms in
the sum for $[D_{11}]$. Since the weight and the sum of rotation numbers are $0$ we have $[D_{11}](a,q) = $
\[ 
\left(E_{\rot, P_{(\e_+,\e_-)}}+E_{\rot, P_{(\e_-,\e_+)}} -E_{\rot,P_{(\e_0)}}\right)(\frac{a}{q},-1,q) = \left(\frac{a^{\frac{1}{2}}-a^{-\frac{1}{2}}}{q^{\frac{1}{2}}-q^{-\frac{1}{2}}}\right)^2
\]
Here $E_{\rot,P_{(\e_0)}}(a,b,q) = b+1$ and $E_{\rot, P_{(\e_+,\e_-)}}(a,b,q) = E_{\rot, P_{(\e_+',\e_-')}}(a,b,q) =E_{\rot, P_{(\e_-,\e_+)}}(a^{-1},b,q^{-1}) 
=E_{\rot, P_{(\e_-',\e_+')}}(a^{-1},b,q^{-1})$ was computed in section \ref{sec.Ehrhart}. We have $E_{\rot, P_{(\e_+,\e_-)}}(\frac{a}{q},-1,q) = 
\frac{a^{-1}-1}{(q^{\frac{1}{2}}-q^{-\frac{1}{2}})^2}$

Likewise since $\lag \e_+, \e_- \rag = -\lag \e_-, \e_+ \rag  = \frac{1}{2}$ we have $[D_{01}](a,q) = $
\[ \left(q^{-\frac{1}{2}}E_{\rot, P_{(\e_+,\e_-)}}+q^{\frac{1}{2}}E_{\rot, P_{(\e_-,\e_+)}}\right)(\frac{a}{q},-1,q) =\]
\[
=\frac{(a^{\frac{1}{2}}-a^{-\frac{1}{2}})(a^{\frac{1}{2}}q^{\frac{1}{2}}-a^{-\frac{1}{2}}q^{-\frac{1}{2}})}{(q^{\frac{1}{2}}-q^{-\frac{1}{2}})^2}
\]

and finally $\lag \e_+', \e_-' \rag = 0$ so $[D_{00}](a,q) = $

\[\left(q^{-1} E_{\rot, P_{(\e_+,\e_-)}}+E_{\rot, P_{(\e_+',\e_-')}}+E_{\rot, P_{(\e_-',\e_+')}} +qE_{\rot, P_{(\e_-,\e_+)}}\right)(\frac{a}{q},-1,q)\]
\[
=\frac{(a^{\frac{1}{2}}-a^{-\frac{1}{2}})(a^{\frac{1}{2}}q^{\frac{1}{2}}-a^{-\frac{1}{2}}q^{-\frac{1}{2}})(q^{\frac{1}{2}}+q^{-\frac{1}{2}})}{(q^{\frac{1}{2}}-q^{-\frac{1}{2}})^2}
\]

Adding up we get the following value for the Hopf link: 
\[P_{1}(H;a,q)=-\frac{(a^{\frac{1}{2}}-a^{-\frac{1}{2}})}{(q^{\frac{1}{2}}-q^{-\frac{1}{2}})^2}(a^{-\frac{3}{2}}-a^{-\frac{1}{2}}(q-1+q^{-1}))\]
in perfect agreement with the $2$-braid formula from the introduction for $c=2,r=1$.

\section{The maximal degree of the colored HOMFLY polynomial}

In this section we take a closer look at the terms in the state sum to prove our bounds on the maximal degree in both $a$ and $q$. Since the colored HOMLFY
polynomial not exactly a polynomial but rather a rational function of $q$ we should perhaps clarify the meaning of $\Dq$. For any rational function we can consider
its Laurent series at infinity. It has a finite highest degree term and its exponent is what we call $\Dq$.

\subsection{Bound on the $a$-degree}

Fix a state $i,\e$ in the state sum. We will compute its $\Da(i,\e)$. Since several states can cancel $\max_{i,\e} \Da(i,\e)$ yields an upper bound
for $\Da P_r(D;a,q)$.

First, the term in front of the state sum contributes $a^{\frac{-r}{2}(c_+-c_-)}$. Next we get $a^{-\frac{1}{2}\sum_{j}\rot(\e_j)}$ from the MOY state sum.
This term is computed in the following lemma:

\begin{Lemma}
\label{Lem.Rot}
If $\e$ sums to $\ga$ then $\sum_{j}\rot(\e_j) = r(s_+-s_-)$, where $s_+$ and $s_-$ are the numbers of positive and negative Seifert circles of the diagram.
\end{Lemma}
\begin{proof}
We say two elementary flows $\alpha,\beta$ intersect if $\alpha$ contains edges in more than one region complementary to $\beta$. Suppose our state $\e$ contains a pair of intersecting flows, $\alpha,\beta$. We may create a new state $\e'$ by replacing two intersecting components by their resolution as in Seifert's algorithm. This new state has the same number of
components and the same rotation numbers. Hence we may compute the sum $\sum_{j}\rot(\e_j)$ on a state $\e$ without intersecting flows. In such states all elementary flows are
parallel to the Seifert-circles of the underlying knot diagram. Around each Seifert circle there are $r$ components of elementary flows with rotation number equal to the sign of the Seifert circle.
Therefore the sum equals $r(s_+-s_-)$.
\end{proof}

Finally we need to estimate $\Da E_{\rot,P_\e}(\frac{a}{q},-1,q)$. According to Lemma \ref{Lem.EhrhartBounds} this is the maximum of $\rot$ over the vertices of the order polytope $P_{\e}$.
This is the sum of the flows with positive rotation number. For states where each flow has a single component this is maximal and equal to $rs_+$.

Adding everything up we find $\Da P_r(D;a,q) \leq \frac{r}{2}(-c_+ + c_- - s_+ -s_- )$ proving part a) of theorem \ref{Thm.Bounds}.

As claimed in the introduction this agrees with Morton's bound when $r=1$. To properly compare the formulas we note that we do not divide by the unknot; our colored HOMFLY is unreduced. Also
Morton's variable $v$ is our $a^{-\frac{1}{2}}$ and his variable $z$ is our $q^{\frac{1}{2}}-q^{-\frac{1}{2}}$.

\subsection{Bound on the $q$-degree}

As with the $a$-bound our strategy is to estimate $\max_{i,\e} \Dq (i,\e)$ by studying the states $(i,\e)$ carefully. 

First the framing term in front of the state sum gives $q^{\frac{-(c_+-c_-)r(r-1)}{2}}$. Next the choice $i$ of resolving the crossings yields
$q^{-\frac{\sum_{c\in C}\sigma(c)i_c}{2}}$. According to Lemma \ref{Lem.Rot} the rotation factor gives $q^{\frac{r(s_+-s_-)}{2}}$. 
Next comes the weight $q^{-w(\e)}$ and finally the Ehrhart polynomial. For the latter we use Lemma \ref{Lem.EhrhartBounds} to get a maximal contribution
of $q^0$ since each order polytope contains the origin. 

Now assume that there exists a state $\e$ attaining the minimal possible weight $w(\e)$. Combining all the contributions and the upper bound for Ehrhart this means
that every crossing $c$ contributes at most 
\[\frac{1}{2}((r-i_c)i_c+(r-i_c)r-\sigma(c)(i_c+r(r-1)))\]
This is maximal for $i_c\in \{0,1\}$ for negative $c$ and for positive $c$ only when $i_c=0$. We get $\frac{1}{2}((r^2-\sigma(c)r(r-1))$.
In conclusion our upper bound is:
\[
\Dq P_r(L;a,q) \leq  r\frac{s_+-s_- + c_+ + c_-(2r-1)}{2}
\]
proving part $q$ of Theorem \ref{Thm.Bounds}.

The big problem with sharpening this bound is that in the presence of a negative crossing $c$ both $i_c =0$ and $i_c=1$ can
yield the maximal $q$-degree and these terms automatically come with opposite signs so that cancellation is likely to occur.
For positive diagrams we have a better chance as we will see in the next subsection.

\subsection{Lower bound for positive diagrams}

For positive diagrams, i.e. $c_-=0$ we derive the lower bound \[\Dq P_r(L;a,q) \geq  r\frac{s + c_+}{2}\]
 announced in Theorem \ref{Thm.Bounds} part p). 
We believe this bound is actually sharp: there are no terms with higher $q$-degree but a proof of this would require more control over the Ehrhart polynomials involved.

To prove our lower bound we show that after expanding $P_r(D;a,q)$ as a Laurent series in $q$, the coefficient of the monomial $a^{\frac{r}{2}(-c_+ - s_+ -s_-)}q^{\frac{r}{2}(s+c_+)}$ is non-zero.
Notice that by Theorem \ref{Thm.Bounds} part a) this is actually the maximal possible degree in $a$. Consider the coefficient $f(q)$ of $a^{\frac{r}{2}(-c_+ - s_+ -s_-)}$. We claim
that $\Dq f(q) = \frac{r}{2}(s+c_+)$. 

To prove this, we note that the analysis of the previous subsection leading to the upper bound on the overall $q$-degree can also be used to compute $\Dq f(q)$. The only difference is that
we get complete control over the Ehrhart term. In order to contribute maximally to the $a$-degree we must have $\Dq E_{\rot,P_\e}(q,aq^{-1},-1) = -\max_{v\in P_\e} \rot(v) = -rs_+$.
Also this joint top coefficient in $q$ and $a$ has coefficient $1$. Setting $i_c=0$ and making all flows non-intersecting we can actually attain the rest of the upper bound in $q$. Therefore 
\[\Dq f(q) \leq  r\frac{-s_+-s_- + c_+}{2}\]
Our final task is to show that these terms do not cancel out. For such states $i,\e$ the $i$ is constant $0$ and the $\e$ consists of elementary flows parallel to the Seifert circles
such that the absolute value of the rotation number equals the number of its components. The order is determined except for flows that are not adjacent. All such flows contribute to our top term
with sign $(-1)^{|\e|}$. Recall that $|\e|$ the length of the sequence $\e$, i.e. the number of elementary flows in the state. The principle of inclusion-exclusion then determines that
the total coefficient must be $\pm 1 \neq 0$ as required.

\section{The head of the colored HOMFLY polynomial of a positive braid}
For closed positive braids we can actually compute the first $r$ coefficients. We may assume the braid has at least two crossings and each generator is used at least twice.

Our strategy is to reduce the computation to braids of index $2$. Indeed one can give a bijection between the state sum of a negative braid and the sliced braid.
If $b=\prod_j \sigma_{k_j}$ is the braid then the sliced braid is $sb=\prod_{j} \sigma_{2k_j-1}$

\begin{Lemma}
Only states where all flows have a single component contribute to the first $r$ coefficients of $P_r(K;a;q)$. Here $K$ is a closure of a positive braid.
\end{Lemma}
\begin{proof}
Suppose a flow has two components $x$ and $x'$, necessarily disjoint. Since there are at least two crossings between each pair of braid lines, there must be at least four
paths $p_1,..,p_4$ connecting $x$ to $x'$ in the resolved braid diagram with the following property. Let $d_j$ be the set of diagonal edges 
used by $p_j$. We require the four sets $d_j$ to be disjoint. Let $i^j_m$ be the maximal crossing parameter found on the path $p_j$. On each path
$p_j$ we can choose $r-i_m^j$ different sequences of elementary flows such that the $s$-th edge of the path is contained in the $s$-th flow.
Since the flows along the path cannot be ordered linearly, the degree must drop by at least $\frac{1}{4}$.
So in total the degree will drop $\frac{1}{4}\sum_{j=1}^4 r-i_m^j$. The crossing parameters also make the degree drop by at least $\sum_{j=1}^4 i_m^j$.
So in conclusion the degree drops by at least $r$.
\end{proof}

From now on all elementary flows in all states will be assumed to have one component unless stated otherwise. 
We will now describe an algorithm to monotonically improve the weight of a given state. A step consists of selecting
two consecutive flows $\alpha,\beta$ in the linear order of the state. After resolving the intersections between the two
we set $\alpha'$ to be everything to the left and $\beta'$ to be everything to the right. The new state is the same as the old
where $\alpha,\beta$ is replaced by $\alpha',\beta'$ in that order. If $\alpha\neq\alpha'$ the weight is improved by at least $1$.

The algorithm terminates on states where no flows cross and touching flows are ordered such that the left is the smaller. These are
the ones that contribute to the highest $q$-degree. 

To understand the contributions to the highest $r$ terms in $q$ we need only look at those states that can be reached from a terminal state by
$r$-steps backwards. Since each step costs at least $1$ such states cannot differ greatly from the terminal states.

\begin{Lemma}
There is a weight preserving bijection from states contributing to the highest $r$ terms of $b$ and states contributing to the highest $r$
terms of $sb$.
\end{Lemma}
\begin{proof}
There is a bijection between the crossings of $b$ and $sb$ and hence also for the crossing variables $i_c$ and corresponding signs and powers of $q$.

Suppose we have such a state for $b$. Resolve the crossings of $sb$ in the same way and create flows on the MOY graph for $sb$ as follows.
If a flow runs between braid lines $k,k+1$ we make the exact same part of the flow between braid lines $2k-1,2k$ of $sb$. Closing and connecting
missing parts this constucts a set of flows for $sb$ adding to the correct total flow on the MOY graph. The order is determined as follows:
two flows between higher braid lines are always higher in the order. Flows within the same two braid lines inherit their order from their order in
the state for $b$. 

This map is especially interesting in case the flows never touch more than three braid lines as is the case for the states that contribute to the first $r$ terms.
In this case the map preserves the weight of the states.

The inverse map is only well defined on states of $sb$ for which everything in the $2k-1,2k$ braid lines is higher than that in the $2k'-1,2k'$ braid lines for $k>k'$.
Luckily this is the case for our contributing states for $sb$. The flows on the $2k-1,2k$ braid lines are ordered linearly and so are those between $2k+1,2k+2$ braid lines.
Matching up these linear orders defines a bijection between the flows. Doing this for all $k$ yields a well defined state for $b$.
Again the weight is preserved and this map is the inverse of the other by construction.
\end{proof}

\section{Proof of the MOY expansion theorem}
\label{Sec.Reform}

In this section we derive our symmetric state sum from the original MOY state sum.

\subsection{The original MOY state sum}
Let us briefly recall the original MOY state sum for MOY graphs $(\Ga,\ga)$ \cite{MOY98}. 
Define an $N$ element set as follows \[A_N = \{-\frac{N-1}{2}, -\frac{N-3}{2},\hdots, \frac{N-1}{2}\}\]
A MOY-state is a function $\sigma: E(\Ga) \to 2^{A_N}$ such that $|\sigma(e)|=\ga(e)$ for all edges $e$ and at each vertex the union of the values of $\sigma$ on the incoming edges equals the union of the values of $\sigma$ on the outgoing edges. Given a state $\sigma$ the MOY-weight of a vertex $v$ is defined as follows. Let $v_L$ and $v_R$ be the left and right edges with respect to the orientation on $\Ga$ then define the weight $\mathrm{wt}(\sigma,v)$ to be
\[\mathrm{wt}(\sigma,v) = q^{\frac{1}{4}\{(x,y)\in \sigma(e_L)\times\sigma(e_R)|x<y\}-\frac{1}{4}\{(x,y)\in \sigma(v_L)\times\sigma(v_R)|x>y\}}\] 
Finally the rotation number of a state $\sigma$ is defined as follows. Replace each edge $e$ by $\ga(e)$ parallel copies each colored by an element of $\sigma(e)$. By the above requirements we can connect edges colored by the same element of $A_{N}$ in a unique planar way to form a system of oriented simple closed curves. Each curve $C$ is colored by a single element $\sigma(C) \in A_{N}$. Define its rotation number $\rot(C)$ to be $1$ if $C$ is oriented counter-clockwise and $-1$ if $C$ is oriented clockwise. Now define \[\rot(\sigma) = \sum_C \sigma(C)\rot(C)\] where the sum is over all closed curves we created.
The MOY state sum is now

\begin{Definition}
\[\lag\Ga,\ga\rag_N(q) = \sum_{\sigma}q^{\rot(\sigma)}\prod_{v\in V(\Ga)}\mathrm{wt}(\sigma,v)\]
\end{Definition}

The MOY state sum was used in \cite{MOY98} to express the anti-symmetric $sl_N$ specializations of the colored HOMFLY $P_{r^t}(D;q^N,q)$.
Here $r^t$ denotes the Young diagram with on column of $r$ boxes, the transpose of the $1$-row diagram.

\begin{Theorem} \cite{MOY98} \\
\label{Thm.MOY}
Let $L$ be a link with oriented link diagram $D$ and $C$ the set of its crossings.
Expand the diagram as a linear combination of MOY graphs $D_i$ by replacing all the crossings as shown in Figure \ref{fig.Di}, then
$P_{r^t}(D;q^N,q) =$ 
\[(q^{-\frac{rN}{2}} q^{\frac{r(r-1)}{2}})^{c_+-c_-} 
\sum_{i\in \{0,1,\hdots r\}^{C}}\left(\prod_{c\in C}(-1)^{i_c} q^{\sigma(c)\frac{i_c}{2}}\right)\lag D_i\rag_N(q)\]
where the sum runs over $i = (i_c)_{c\in C}$ and $\sigma(c)$ is the sign of the crossing $c$.
\end{Theorem}

This implicitly determines the polynomial $P_{r^t}(L;a,q)$ and hence also $P_r(L;a,q) = (-1)^rP_{r^t}(L;a,q^{-1})$. Here we used
the general symmetry formula for the $\lambda$-colored HOMFLY polynomial: $P_{\lambda}(L;a,q)=(-1)^{|{\lambda}|} P_{{\lambda}^t}(L;a,q^{-1})$
where $\lambda^t$ denotes the transposed partition and $|\lambda|$ is the number of boxes \cite{Zh13}.

However we cannot yet use this formula to transform the MOY state sum into a state sum for $P_r$ since the $a$ dependence is too implicit.
In the next section we will make it explicit.

\subsection{Reformulation of the MOY state sum}

We now reformulate the MOY state sum to explicitly include the variable $a$. The basic idea is to collect the MOY states with equal rotation numbers and vertex weights and sum them explicitly, giving expressions in $a$. Those explicit sums can be written as sums over lattice points in polytopes bringing us back to the $q$-Ehrhart polynomials. The auxilliary variable $b$ included in the Ehrhart polynomials will happily turn out cancel out for all links.

We begin by reinterpreting the MOY states as functions $\sigma: \mathcal{E} \to 2^{A_N}$ on the set $\mathcal{E}$ of elementary flows. Starting with a MOY state and an elementary flow $\e$ define 
$\sigma(\e)$ to be the set of all $x\in A_N$ such that there is no larger elementary flow such that the MOY state assigns $x$ to each of its edges. Conversely a function $\sigma$ determines a MOY state
by assigning to each edge $e$ the subset $\bigcup_{\e: \e(e)=1}\sigma(\e)$. It is easy to see that such functions are indeed in bijection with the MOY states.
In what follows these sets will be identified without writing out the above bijection explicitly.

Next we collect many states with the same weight by noting that the weight of a state is already determined by the relative sizes of the function values.
More precisely, a MOY state $\sigma$ determines a sequence of elementary flows $\e$ by as follows. The union $I_\sigma=\bigcup_{f \in \mathcal{E}}\sigma(f)$ is a subset of $A_N$ and is therefore
ordered. For $j\leq |I_\sigma|$ define $\e_j$ to be the inverse image under $\sigma$ of the $j$-th element of $I_\sigma$. 
The point of all this is that for a MOY state $\sigma$ inducing a sequence $\e$ we have $\prod_{v\in V(\Ga)}\mathrm{wt}(\sigma,v) = q^{w(\e)}$.

We now want to sum all the MOY states inducing the same sequence $\e$. This is where the $q$-Ehrhart polynomials come in. 

\begin{Lemma}
\label{Lem.Collect}
\[\sum_{\sigma\text{ inducing }\e}q^{\rot \sigma} = (-1)^{|\e|}(qa)^{-\frac{1}{2}\sum_{j} \rot(\e_j)}E_{\rot, P_{\e}}(aq,-b-1,q^{-1})|_{a=q^N,\ b=N}
\]
\end{Lemma}
\begin{proof}
The left hand side is equal to
\[
\mathcal{LHS} = \sum_{v_j \in A_N^{|\e|}:i<j\Rightarrow v_i<v_j}q^{\rot x}
\]
Shifting everything to the right positive numbers we first get 

\[
\mathcal{LHS} = q^{-\frac{N+1}{2}\sum_{j} \rot(\e_j)}\sum_{v \in \{1,\hdots,N\}^{|\e|}:i<j\Rightarrow v_i<v_{j}}q^{\rot x}
\]
This is precisely a sum over the interior of the scaled order polytope $(N+1)P_{\e}$ and hence by $q$-Ehrhart reciprocity, Theorem \ref{q.Ehrhart}b, we have
\[\mathcal{LHS} = (-1)^{|\e|}q^{-\frac{N+1}{2}\sum_{j} \rot(\e_j)} E_{\rot,P_{\e}}(q^{N+1},-N-1,q^{-1})
\]
\end{proof}

To summarize our reformulation so far let us define a new anti-symmetric evaluation of a MOY graph. 

\begin{Definition}
The anti-symmetric evaluation $\lag\Ga,\ga\rag(q,a,b)$ of a MOY graph $(\Ga,\ga)$ is defined as
\[
\lag\Ga,\ga\rag(q,a,b) = \sum_{\e: \sum_j \e_j = \gamma} (-1)^{|\e|}(qa)^{-\frac{1}{2}\sum_{j}\rot(\e_j)}
q^{w(\e)} E_{\rot,P_{\e}}(aq,-b-1,q^{-1})
\]
Here the sum is over sequences $\e=(\e_1,\e_2\hdots)$ of elementary flows and $P_\e$ is the order polytope
of $\e$ interpreted as a linearly ordered set $\e_1\leq\e_2\leq \hdots$. The length of a sequence $\e$ is denoted $|\e|$ 
and finally the weight is $w(\e) = \sum_{i<j}\lag \e_i,\e_j \rag$.
\end{Definition}

What we have done so far is prepare the proof of the following Theorem.

\begin{Theorem}
\label{Thm.AS}
We have $\lag\Ga,\ga\rag_N(q) = \lag \Ga,\ga \rag(q,a,b)|_{a=q^N,\ b=N}$ and hence the following state sum for the anti-symmetric colored HOMFLY polynomial:
$P_{r^t}(D;a,q) =$ 
\[(a^{-\frac{r}{2}} q^{\frac{r(r-1)}{2}})^{c_+-c_-} 
\sum_{i\in \{0,1,\hdots r\}^{C}}\left(\prod_{c\in C}(-1)^{i_c} q^{\sigma(c)\frac{i_c}{2}}\right)\lag D_i\rag(q,a,b)\]
where the sum runs over $i = (i_c)_{c\in C}$ and $\sigma(c)$ is the sign of the crossing $c$.
\end{Theorem}
\begin{proof}
First we interpret the states in the MOY state sum as functions $\sigma:\mathcal{E}\to 2^{A_N}$. 
Next we collect all states inducing the same sequence of elementary flows $\e$ using Lemma \ref{Lem.Collect}.
The state sum is a direct corollary of the MOY theorem \ref{Thm.MOY}.
\end{proof}

Although not strictly necessary it is nice to know that the terms involving the extra variable $b$ actually cancel out. 
We can thus set its value to whatever we like, $b=0$ makes intermediate calculations easier. 
A good way to understand this is to work with knotted versions of MOY graphs as in \cite{GV13}. It was shown there that
for positive MOY graphs diagrams the evaluation depends only on $a$ and $q$. A positive MOY graph diagram is a diagram 
where all Seifert circles have positive rotation number. Since our evaluation is an invariant of knotted MOY graphs the independence of $b$
follows from the following lemma.

\begin{Lemma}
\label{Lem.Alexander}
Any knotted framed MOY graph $\Gamma$ has a diagram with only positive cycles.
\end{Lemma}
\begin{proof}
One can simply run Alexander's algorithm on a piecewise linear version of $\Gamma$, see for example \cite{Ma04}. This turns every edge into a positive edge and we can arrange it such that the cyclic order at the vertices stays intact. 
\end{proof}

To conclude our reformulation of the MOY state sum we now return to the HOMFLY symmetry formula $P_r(L;q,a) = (-1)^rP_{r^t}(L;q^{-1},a)$ \cite{Zh13} mentioned earlier.
By Theorem \ref{Thm.AS} we can replace $q$ by $q^{-1}$ and multiply by $(-1)^r$ in the whole state sum to get the symmetric state sum for the colored HOMFLY polynomial
$P_r(q,a,b)$ promised in Section \ref{Sec.Sym}. This concludes the proof of Theorem \ref{Thm.SymMOY}.

\section{Discussion}

We end this paper by some speculations, future directions and challenges.

It is natural to seek bounds on the degree of the more general colored HOMFLY polynomial depending on any partition, not just the rows or columns.
Such estimates will be more difficult but perhaps in families of scaled partitions one can still say something about $P_{r\lambda}$ for any $r$.

One may also wonder if there is an analogous way to treat alternating knots. For this one would have to expand the positive crossings in a different way.

Since MOY graphs appear naturally in categorification of the $sl_N$ invariants the results in this paper can probably be generalized to the knot homology level.

Perhaps the present expansion of the colored HOMFLY in terms of Ehrhart polynomials can be used to settle the question of q-holonomicity. For this one would need to improve control over the weight $w$ and also gain a better understanding of the behaviour of the Ehrhart polynomial of a growing polytope.

\bibliography{mybib}{}

\bibliographystyle{hamsalpha}
\end{document}